\documentclass[a4paper,11pt]{article}

\usepackage{xcolor}
\usepackage{array}
\usepackage{cancel}
\usepackage[british]{babel}
\usepackage{pifont,subfigure,graphicx}
\usepackage{caption}

\usepackage[colorlinks, linkcolor=black,citecolor=blue,urlcolor=blue]{hyperref}
\usepackage{url}

\usepackage{amsthm,amsmath,amssymb,amsfonts,mathrsfs,amsfonts,amstext,amsopn}
\usepackage{mathtools}
\usepackage{tikz}
\usepackage[normalem]{ulem}
\usepackage{hhline}
\usepackage[version=3]{mhchem}
\usepackage{bbm}

\usepackage{authblk}

 \usepackage{microtype}

\newtheorem{theorem}{Theorem}[section]
\newtheorem{lemma}[theorem]{Lemma}
\newtheorem{proposition}[theorem]{Proposition}
\newtheorem{corollary}[theorem]{Corollary}

\theoremstyle{definition}
\newtheorem{remark}[theorem]{Remark}
\newtheorem{example}[theorem]{Example}

\newcommand{\R}{\mathbb R}
\newcommand{\N}{\mathbb N}
\newcommand{\cC}{\mathcal C}
\newcommand{\cN}{\mathcal N}
\newcommand{\cR}{\mathcal R}
\newcommand{\norm}[1]{\lVert #1\rVert}

\begin{document}

\title{A Proof of the Global Attractor Conjecture in a Special Case}
\author{Carsten Wiuf}
\affil{Department of Mathematical Sciences\\ University of Copenhagen}

\date{\today}

\maketitle

\begin{abstract}
We prove the Global Attractor Conjecture for   complex balanced mass-action reaction networks whose reachable siphons satisfy two structural conditions, allowing multiple linkage classes. The proof proceeds in two steps. A structural condition relating the stoichiometric space to the reactions active within a boundary face ensures that a stoichiometric compatibility class contains at most one boundary equilibrium with any prescribed zero set. Finiteness of the possible zero sets and connectedness of the $\omega$-limit set then imply that any boundary limit set consists of a single equilibrium. To exclude convergence to such an equilibrium, we consider the embedded reaction network obtained by projecting onto the vanishing species and freezing the concentrations of the surviving species at a positive limit. The structural condition guarantees complex balance of this embedded network, while a further condition on its   minimal active linkage classes ensures that an explicit Chetaev function is strictly increasing near the boundary. This excludes the boundary point as an accumulation point whenever the surviving concentrations converge. Consequently, every positive trajectory converges to the unique positive equilibrium in its stoichiometric compatibility class. Examples illustrate the hypotheses and their relation to strong endotacticity.
\end{abstract}

\section{Introduction}

Let $\cN$ be a mass-action reaction network that is \emph{complex balanced} for its rate
constants, and let $x^*$ be the unique positive complex balanced equilibrium in a given
stoichiometric compatibility class. The Global Attractor Conjecture (GAC) asserts that every
positive trajectory in that class converges to $x^*$~\cite{HornJackson1972,CraciunDickensteinShiuSturmfels2009}.
Write $x=(x_1,\ldots,x_m)$. The Horn--Jackson entropy
\[
 V(x)=\sum_{i=1}^m\Bigl[x_i\log\frac{x_i}{x^*_i}-x_i+x^*_i\Bigr]
\]
is a strict Lyapunov function in the positive orthant: $\dot V\le0$, with equality at a positive
point only if that point is itself complex balanced, hence only at $x^*$. What $V$ does not control
is the boundary, and this is the entire content of the conjecture. GAC has been proved in special
cases: for reaction networks with a single linkage class~\cite{Anderson2011}, for strongly endotactic
reaction networks~\cite{GopalkrishnanMillerShiu2014}, for reaction networks whose stoichiometric subspace has
dimension at most three~\cite{Pantea2012}, and for first-order endotactic reaction
networks~\cite{Xu2026}; see also~\cite{AngeliDeLeenheerSontag2007,JohnstonSiegel2011,AndersonShiu2010}
for persistence criteria.

The Lyapunov function   yields a clean dichotomy, which we record below in Proposition~\ref{prop:dichotomy}, because it emphasises  what
has to be excluded. Note that  sub-level sets of $V$ in the closure of a stoichiometric compatibility class are compact.
Consequently, the positive trajectory is bounded, and its $\omega$-limit set is
non-empty, compact, connected and invariant. Throughout, $\omega(x_0)$ denotes the $\omega$-limit set of the positive
trajectory starting at $x_0$.

\begin{proposition}[Siegel and MacLean~\cite{SiegelMacLean2000}, Theorem~3.2; see also Sontag~\cite{Sontag2001}]
\label{prop:dichotomy}
Let $\cN$ be a complex  balanced reaction network and $x_0>0$. Then, exactly one of the following holds:
\begin{enumerate}
\item[(i)] $\omega(x_0)=\{x^*\}$, that is $x(t)\to x^*$; or
\item[(ii)] $\omega(x_0)\subseteq\partial\R^m_{\ge0}$, and every point of $\omega(x_0)$ is an
equilibrium.
\end{enumerate}
\end{proposition}

Alternative (i) is the conclusion of GAC. Alternative (ii) allows the limit set to contain more than one boundary equilibrium. Thus, excluding convergence to an individual boundary equilibrium does not by itself establish global attraction.

Our main result, Theorem~\ref{thm:GAC}, resolves this difficulty under two structural conditions. The first concerns the relation between the stoichiometric space and the reactions active on a boundary face. It ensures that a stoichiometric compatibility class contains at most one boundary equilibrium with any prescribed zero set. Since there are only finitely many possible zero sets, connectedness of the limit set implies that every boundary limit set consists of a single equilibrium (Corollary~\ref{cor:singleton}).

The second condition concerns the minimal active linkage classes of an embedded reaction network associated with the vanishing species. It yields a lower bound on the dissipation of an explicit Chetaev function and excludes convergence to an individual boundary equilibrium. The first structural condition also guarantees the embedded reaction network  is  complex balanced, which is needed for this local argument. Combining the two results proves global attraction. The precise conditions are stated below, after the necessary notation has been introduced.

For the local argument, suppose, for contradiction, that a positive trajectory approaches
 a boundary point.  Write the vanishing coordinates of the point as $z$ and the non-zero (surviving) coordinates
as $s$, so $x=(z,s)$ and
\begin{equation}\label{eq:zs}
 z(t_n)\to0,\qquad t_n\to\infty, \qquad s(t)\to b>0,\qquad t\to\infty,
\end{equation}
where $(t_n)_n$ is an increasing sequence.
Let $\Sigma$ denote the indices of the vanishing coordinates, and consider the embedded reaction network $\cN_\Sigma(b)$ (projection onto $\Sigma$) with fixed (\emph{frozen}) surviving coordinates $b$.
If $\cN_\Sigma(b)$ has a positive complex balanced equilibrium and satisfies the
class-wise condition defined in Section~\ref{sec:assumptions}, then the
Chetaev function increases near the boundary, excluding the point as an accumulation point.

\begin{theorem}
\label{thm:boundary}
Let $\Sigma$ be a non-empty siphon of a weakly reversible mass-action reaction network, write
$x=(z,s)$ as above, and assume that a positive trajectory satisfies $s(t)\to b>0$ as
$t\to\infty$.  Suppose that the embedded reaction network frozen at $s=b$
admits a positive complex balanced equilibrium $c$, and that each (so-called) \emph{minimal} active embedded linkage class has a non-zero non-negative stoichiometric vector (condition \emph{($\diamond$)} of Section~\ref{sec:assumptions}).  Then,
\[
 \liminf_{t\to\infty}\norm{z(t)}>0 .
\]
In particular, $(0,b)$ is not an $\omega$-limit point of the trajectory, and \eqref{eq:zs} is impossible.
\end{theorem}

\noindent\textbf{Main result (Theorem~\ref{thm:GAC}).} For a complex  balanced reaction network, suppose that two structural conditions hold at every relevant boundary face: the reactions active on the face account for all stoichiometric directions within it, and each minimal active linkage class of the corresponding embedded network satisfies a non-negative stoichiometric condition. Then every positive trajectory converges to the unique positive equilibrium in its stoichiometric compatibility class.

\medskip

The   terminology will be introduced in Section~\ref{sec:assumptions}. A proof of Theorem~\ref{thm:boundary} is   in Section~\ref{sec:th1.1}.
Note that \eqref{eq:zs} is \emph{not} assumed in Theorem~\ref{thm:boundary}; only the convergence of
the surviving coordinates is.  Lemma~\ref{lem:Sz} and Remark~\ref{rem:R1} do use \eqref{eq:zs}, but they
are not needed for the theorem: they merely supply a sufficient condition for ($\diamond$).
The reaction network in Theorem \ref{thm:boundary} needs only be weakly reversible for the   statement to hold. However, the embedded reaction network is required to be complex balanced.
Section~\ref{sec:frozen} studies boundary equilibria and gives sufficient conditions for complex balance of the frozen embedded reaction network. Section~\ref{sec:gac} shows that the first structural condition reduces any boundary limit set to a single equilibrium (Corollary~\ref{cor:singleton}), then combines this with Theorem~\ref{thm:boundary} to prove global attraction (Theorem~\ref{thm:GAC}). Section~\ref{sec:examples} provides examples.

The conditions used here are related to endotacticity, but do not reduce to it.
Gopalkrishnan, Miller and Shiu~\cite{GopalkrishnanMillerShiu2014} prove GAC for complex balanced reaction networks that are \emph{strongly endotactic}, a requirement on the   reaction network imposed on
every vector $w\in\R^m$ that is not orthogonal to the stoichiometric space.  Condition
($\diamond$) is instead imposed on the embedded reaction network   attached to the siphons that a
compatibility class can actually reach, and only for vectors in the open positive orthant of
$\R^{|\Sigma|}$.  Moreover, hypothesis (a) of Theorem~\ref{thm:GAC} has no counterpart in \cite{GopalkrishnanMillerShiu2014}.  The two
sets of hypotheses are incomparable in both directions: Example~\ref{ex:notse} is weakly
reversible and complex balanced, satisfies the hypotheses of Theorem~\ref{thm:GAC}, and is not
strongly endotactic, while Example~\ref{ex:incomparable} gives a strongly endotactic reaction network for
which the hypotheses of Theorem~\ref{thm:GAC} fail.  The proof below is also of a different kind, using only   an explicit Chetaev function in place of differential inclusions.

\subsection*{Acknowledgement}

My initial idea was to use complex balance of the embedded reaction network to construct a Lyapunov-like or a Chetaev-like function to conclude the GAC validity in some cases of interest. Early on, I   engaged two AI agents from different companies to check examples and go through tedious computations. My conversations with the two AI agents intensified as the research progressed, in particular after I led the agents comment on the other agent's work. The main ideas and simplifications are mine. The AI agents drafted the paper which I completely rewrote. In this process, the AI agents assisted as proof checkers and were allowed to suggest rewordings.
This paper will only appear on arXiv.

\section{Set-up}

Let $\cN=(\cC,\cR)$ be a labelled multi-digraph encoding  an $m$-dimensional  mass-action reaction network with ODE system
\begin{equation*}
 \dot x=\sum_{r\in\cR} k_r x^{y_r}(y'_r-y_r),
 \qquad x=(x_1,\ldots,x_m)\in\R^m_{>0},
\end{equation*}
complexes $y_r,y'_r\in\cC\subseteq \N_0^m$,  and labelled reactions
$$r\colon y_r \stackrel{k_r}{\longrightarrow} y'_r\in \cR,$$
where $k_r>0$ is a positive rate constant (multiple edges between the same two complexes are allowed). The vector $(k_r)_{r\in\cR}$ is referred to as the rate vector. Standard notation $x^n=\prod_{i=1}^m x_i^{n_i}$, $n=(n_1,\ldots,n_m)\in\N_0^m$, is adopted.  We write $u\circ v$ for the component-wise product of two vectors.
See Horn and Jackson~\cite{HornJackson1972} and Feinberg~\cite{Feinberg1987,Feinberg2019} for foundational references, and Anderson and Kurtz~\cite{AndersonKurtz2015} for the stochastic counterpart.

The reaction network $\cN$ is \emph{weakly reversible} if the connected components of the multi-digraph are strongly connected.

Let $\Sigma\subseteq\{1,\ldots,m\}$ be a subset of the indices.    Write
\begin{equation}\label{eq:zeq}
 x=(z,s),\qquad y_r=(\alpha_r,\beta_r),\qquad  y'_r=(\alpha'_r,\beta'_r),
\end{equation}
where $z=(x_i, i\in \Sigma)$, $s=(x_i, i\in\Sigma^c)$, and $y_r,y'_r$ are partitioned similarly.
We write $\pi_\Sigma(x)=z$ and $\pi_{\Sigma^c}(x)=s$ for the two projections, and use the same
symbols for the corresponding projections of subspaces of $\R^m$.  Then,
\begin{align*}
 \dot z&=\sum_{r\in\cR}k_rz^{\alpha_r}s^{\beta_r}(\alpha'_r-\alpha_r),\qquad
 \dot s=\sum_{r\in\cR}k_rz^{\alpha_r}s^{\beta_r}(\beta'_r-\beta_r).
\end{align*}
For fixed $s>0$, the $z$-equation is the mass-action system encoded by the \emph{embedded reaction network} $\mathcal N_\Sigma(s)=(\cC_\Sigma,\cR_\Sigma(s))$ obtained by projecting  onto the species set $\Sigma$ (called the $\Sigma$-projection).
The embedded reaction network has complexes $ \alpha_r,\alpha'_r$, $r\in \cR$, and reactions
$$\alpha_r\stackrel{k_r(s)}{\longrightarrow}\alpha'_r,\qquad  k_r(s)=k_rs^{\beta_r}>0,\qquad r\in\cR.$$
Equal projected complexes are identified in $\cC_\Sigma$ and parallel reactions remain distinct in $\cR_\Sigma(s)$ (so effectively, $\cN_\Sigma(s)$ is a multi-digraph, even if $\cN$ is not). We index projected reactions by their original reaction labels $r\in\cR$. We call this the embedded reaction network \emph{frozen} at $s$.

\section{Assumptions}\label{sec:assumptions}

Let $\cN$ be weakly reversible. An embedded reaction network of $\cN$ is then also weakly reversible by projection. Furthermore,   let $\Sigma\subseteq \{1,\ldots,m\}$ be a non-empty siphon; by (1) and \cite{AngeliDeLeenheerSontag2007}, the zero set of a boundary limiting point is a siphon. Thus, there are no reactions with  $\alpha_r=0$   and  $\alpha'_r\ne0$.

Let
\[
 S=\operatorname{span}\{\alpha'_r-\alpha_r:r\in\cR\}
\]
be the stoichiometric space of $\mathcal N_\Sigma(s)$.

\begin{lemma}\label{lem:Sz}
Under \eqref{eq:zs},
\[
 S^\perp\cap\R^{|\Sigma|}_{\ge0}=\{0\},
 \qquad z(t)\in S,\qquad t\ge0.
\]
Consequently $S\cap\R^{|\Sigma|}_{>0}\neq\varnothing$.
\end{lemma}

\begin{proof}
 Let $(t_n)_n$ be an increasing sequence, such that $z(t_n)\to0$ for $n\to\infty$.
If $0\neq w\in S^\perp\cap\R^{|\Sigma|}_{\ge0}$, then $w\cdot z(t_n)=w\cdot z(0)>0$ is constant and strictly positive, whereas $z(t_n)\to0$, a contradiction.  Also $z(t)-z(t_n)\in S$ for  all $n$.  Closedness of $S$ gives $z(t)\in S$.  Since $z(t)>0$ and $z(t)\in S$, the last assertion is immediate.
\end{proof}

A  linkage class of $\cN$ is called \emph{$\Sigma$-blind} if one of its complexes has zero $\Sigma$-projection. Then all complexes of that linkage class have zero $\Sigma$-projection.  Indeed, if $\alpha_r=0$, the siphon property forces $\alpha'_r=0$, and by weak reversibility it propagates to all complexes in the class.
All remaining linkage classes of $\cN$ are called \emph{$\Sigma$-active}. Complexes are similarly called $\Sigma$-blind and $\Sigma$-active according to the linkage class they belong to. Projection may identify complexes belonging to different original   linkage classes, so several original   classes may merge into a single linkage class of the embedded network.

Consider the embedded reaction network frozen at $s=b$, $\cN_\Sigma(b)$,   and assume it has a positive complex balanced   equilibrium $ c\in\R^{|\Sigma|}_{>0}$.  A simple sufficient condition is that
$\cN_\Sigma(s)$ has deficiency zero: being weakly reversible, it is then complex balanced for every
positive choice of its rate constants, in particular for the frozen ones. Section~\ref{sec:frozen}
gives conditions that do not require deficiency zero.  The   rate constants  are
\begin{equation*}
 k_r(b)=k_r b^{\beta_r}>0,\qquad r\in\cR.
\end{equation*}
For each  reaction $r\in\cR$, put
\begin{equation*}
 q_r:=k_r b^{\beta_r}c^{\alpha_r}>0.
\end{equation*}
At every embedded complex $a\in\cC_\Sigma$, complex balance gives by definition,
\begin{equation}
\label{eq:sumqr}
 \sum_{r\in\cR:\alpha_r=a}q_r=\sum_{r\in\cR:\alpha'_r=a}q_r.
\end{equation}
Thus, $q_r$, $r \in\cR$, is a positive circulation on the embedded reaction multi-digraph.

If there are no $\Sigma$-active linkage classes, then $\dot{z}=0$, and the conclusion of Theorem 1.2 is immediate. Let $L_1,\dots,L_R$ be the $\Sigma$-active embedded linkage classes, and let
\[
 S_\ell=\operatorname{span}\{\alpha'_r-\alpha_r:r\in L_\ell\}
\]
be the stoichiometric space corresponding to $L_\ell$.  Let $\cC^{\mathrm{act}}_\Sigma\subseteq \cC_\Sigma$ be the set
of $\Sigma$-active embedded complexes, ordered component-wise. 
Henceforth assume that at least one such class exists.  A complex
$a\in\cC^{\mathrm{act}}_\Sigma$ is \emph{$\Sigma$-minimal} if there are no $\widehat a\in\cC^{\mathrm{act}}_\Sigma$
with $\widehat a\le a$ and $\widehat a\neq a$, and a $\Sigma$-active embedded linkage class is \emph{$\Sigma$-minimal} if it
contains a $\Sigma$-minimal complex.  As $\cC^{\mathrm{act}}_\Sigma$ is finite and non-empty, at least one
class is $\Sigma$-minimal.  We impose the following structural condition:
\begin{equation*}
 (\diamond)\qquad
 S_\ell\cap\R^{|\Sigma|}_{\ge0}\neq\{0\}
 \quad\text{for every $\Sigma$-minimal }L_\ell.
\end{equation*}

\begin{remark}\label{rem:R1}
If $R=1$, then $S=S_1$ and $L_1$ is $\Sigma$-minimal.  Lemma~\ref{lem:Sz} gives $S_1\cap\R^{|\Sigma|}_{>0}\neq\emptyset$, hence ($\diamond$).  Thus, ($\diamond$) is precisely the additional class-wise requirement introduced when one passes from one active embedded linkage class to several.  Imposing ($\diamond$) on all $\Sigma$-active classes is strictly stronger than imposing it only on $\Sigma$-minimal classes; see Example~\ref{ex:minimalmatters}.
\end{remark}

\begin{remark}
By Stiemke's alternative~\cite{Stiemke1915},
\[
 S_\ell\cap\R^{|\Sigma|}_{\ge0}=\{0\}
\]
if and only if there is $p\in\R^{|\Sigma|}_{>0}$ with $p\perp S_\ell$.  Thus ($\diamond$) says that no $\Sigma$-minimal embedded linkage class has a strictly positive conservation law of its own.     A  simple sufficient condition is that the class contains two  complexes $a, \widehat a$ with $a\ge \widehat a$ (component-wise) and $a\neq \widehat a$.
\end{remark}

Equivalently, ($\diamond$) holds for $L_\ell$ if and only if for every $\theta\in\R^{|\Sigma|}_{>0}$,
some reaction $r\in L_\ell$, whose source complex minimises $a\cdot\theta$ over $L_\ell$, satisfies
$(\alpha'_r-\alpha_r)\cdot\theta>0$.  Indeed, by the previous remark, ($\diamond$) says that no
$\theta>0$ is orthogonal to $S_\ell$, that is, $a\cdot\theta$ is non-constant on $L_\ell$. As
$L_\ell$ is strongly connected,  a directed path from a minimiser of $a\cdot\theta$ to a non-minimiser contains a reaction $a\to a'$, such that $a$  is a minimiser while $a'$ is not.

With $w=-\theta$, this is the pair of conditions ``$w\cdot y>w\cdot y'$ and $y$ is $w$-maximal''
occurring in the definition of a strongly endotactic reaction network of Gopalkrishnan, Miller and
Shiu~\cite{GopalkrishnanMillerShiu2014}, with maximality taken within each $L_\ell$, rather than over all
complexes, and only for vectors $w\in-\R^{|\Sigma|}_{>0}$.  The remaining requirement of that
definition, endotacticity, holds automatically here, since $\cN_\Sigma$ is weakly reversible. Examples illustrating the difference can be found in Section~\ref{sec:examples}.

\section{A Chetaev-type function}

For a  complex $a\in\cC_\Sigma$ of the embedded reaction network $\cN_\Sigma(b)$, let
\begin{equation*}
 A_a(z)=\frac{z^a}{c^a},
\end{equation*}
where $c>0$ is a complex balanced equilibrium of $\cN_\Sigma(b)$. Define
\begin{equation}\label{eq:U}
 U(z)=\sum_{j\in\Sigma}z_j
 \left(1-\log\frac{z_j}{c_j}\right).
\end{equation}
For $z>0$ sufficiently close to $0$, $U(z)>0$, and $U(z)\to0$ as $z\to0$. We will show that $U$ is a Chetaev-type function, that is, the derivative of $U$ along a trajectory is positive near $0$, $\dot U>0$~\cite{Rumyantsev2001}. Hence, a trajectory cannot approach $(0,b)$.

Multiplying the complex balance identity \eqref{eq:sumqr} by $A_a(z)$ and summing over all embedded complexes $a\in\cC_\Sigma$ yields
\begin{equation}\label{eq:qr}
 \sum_{r\in\cR}q_r A_{\alpha_r}(z)
 =\sum_{r\in\cR}q_r A_{\alpha'_r}(z).
\end{equation}
Indeed, on each side the sum over complexes may be interchanged with the finite sum over reactions.  The \emph{dissipation} of the frozen embedded system is
\begin{align}\label{eq:4.3}
 D(z)
 &=\sum_{r\in\cR}q_r A_{\alpha_r}
 \log\frac{A_{\alpha_r}}{A_{\alpha'_r}}=\sum_{r\in\cR}q_r
 \left[
 A_{\alpha_r}\log\frac{A_{\alpha_r}}{A_{\alpha'_r}}
 -A_{\alpha_r}+A_{\alpha'_r}
 \right]\ge0,
\end{align}
where the second equality follows from \eqref{eq:qr}, and the inequality follows from $u\log(u/v)-u+v\ge0$ for $u,v>0$.

For each linkage class $L_\ell$, set
\begin{align}
 D_\ell(z)
 &=\sum_{r\in L_\ell}q_r
 \left[
 A_{\alpha_r}\log\frac{A_{\alpha_r}}{A_{\alpha'_r}}
 -A_{\alpha_r}+A_{\alpha'_r}
 \right], \label{eq:DA}
 \\
 A_\ell(z)&=\max_{a\in L_\ell}A_a(z).
\nonumber 
\end{align}
The $\Sigma$-blind reactions have $\alpha_r=\alpha'_r=0$ and contribute zero to $D$.  Hence,
\begin{equation*}
 D=\sum_{\ell=1}^R D_\ell.
\end{equation*}
The following technical lemma will be useful.

\begin{lemma}
Let $h(v)=v-1-\log v,$ for $ v>0$.
For $u\ge0$, let $\psi(u)\in(0,1]$ be the inverse of $h$ on $(0,1]$, that is, $h(v)=u$.  Then, $\psi$ is decreasing, $\psi(0)=1$, and
\begin{equation}
 \psi(u)\ge1-2\sqrt u,
 \qquad u\ge0.
 \label{eq:5.1}
\end{equation}
\end{lemma}

\begin{proof}
Since $h'(v)=1-v^{-1}<0$ on $(0,1)$ and $h(1)=0$, then $\psi$ is decreasing and $\psi(0)=1$. Furthermore,
by expanding the logarithm,  $h(v)\ge(1-v)^2/2$ on $(0,1]$. With $v=\psi(u)$ this yields $1-\psi(u)\le\sqrt{2u}\le2\sqrt u$, proving the claim.
\end{proof}

For a strongly connected multi-digraph, the \emph{directed diameter}  is the maximum length of  shortest directed paths between any two complexes. If the graph has one vertex, the directed diameter is zero, otherwise $d_\ell\ge1$.

The next lemma and condition ($\diamond$) are key to obtain a bound on the ratio of $D_\ell$ over $A_\ell$.

\begin{lemma}
\label{lem:uniformprop}
Fix an embedded linkage class $L_\ell$ with directed diameter $d_\ell$.
\[
 q_*:=\min_{r\in L_\ell}q_r,
 \qquad
 \varepsilon:=\frac{D_\ell(z)}{q_* A_\ell(z)}.
\]
If $d_\ell=0$, the conclusion holds trivially. If $d_\ell\ge1$ and $\varepsilon\le(4d_\ell)^{-2}$, then
\begin{equation*}
 1\ge\frac{A_a(z)}{A_\ell(z)}
 \ge1-2d_\ell\sqrt\varepsilon,
 \qquad a\in L_\ell.
\end{equation*}
\end{lemma}

\begin{proof}
Fix $z>0$. Since $L_\ell$ is finite, choose a complex $a_0\in L_\ell$ with
$A_{a_0}(z)=A_\ell(z)$. Thus $B_{a_0}=1$, where $B_a:=A_a/A_\ell\in(0,1]$.
If $d_\ell=0$, the linkage class has one complex and the conclusion is immediate. Hence assume $d_\ell\ge1$.
For each reaction $a\to a'$ in $L_\ell$, non-negativity of the summands in \eqref{eq:DA} gives
\[
 q_* B_a h(B_{a'}/B_a)
 \le \frac{D_\ell}{A_\ell}=q_*\varepsilon.
\]
Thus, $h(B_{a'}/B_a)\le\varepsilon/B_a$. Since $h$ is strictly decreasing
on $(0,1]$, and its value on $[1,\infty)$ is non-negative, then either
$B_{a'}/B_a\ge1$, or by definition of $\psi$ (the inverse of $h$ on $(0,1]$),
\begin{equation}\label{eq:5.3}
 B_{a'}\ge B_a\psi(\varepsilon/B_a). 
\end{equation}
Let $f(t):=t\psi(\varepsilon/t)$ for $t>0$, which is increasing: if
$t_1\le t_2$, then $\psi(\varepsilon/t_1)\le\psi(\varepsilon/t_2)$,
so $f(t_1)\le f(t_2)$. Define
\[
 \gamma_0=1,\qquad \gamma_{i+1}=f(\gamma_i)
 =\gamma_i\psi(\varepsilon/\gamma_i).
\]
We show by induction that every complex reachable from $a_0$ by a path
of length $i$ satisfies $B_a\ge\gamma_i$. For $i=0$, this is
$B_{a_0}=\gamma_0=1$. If the assertion holds for a path ending at $a$,
then \eqref{eq:5.3} and monotonicity of $f$ give, for the next edge $a\to a'$, that
\[
 B_{a'}\ge f(B_a)\ge f(\gamma_i)=\gamma_{i+1}.
\]
Furthermore, $\gamma_i\le1$ for every $i$. If $\gamma_i\ge1/2$,
then $\varepsilon/\gamma_i\le2\varepsilon\le1/8$, so \eqref{eq:5.1} gives
\[
 \gamma_{i+1}\ge\gamma_i-2\sqrt{\varepsilon\gamma_i}
 \ge\gamma_i-2\sqrt\varepsilon.
\]
Inductively, for $0\le i\le d_\ell$,
\[
 \gamma_i\ge1-2i\sqrt\varepsilon\ge\frac12,
\]
using $\varepsilon\le(4d_\ell)^{-2}$. Every complex can be reached
from $a_0$ by a directed path of length at most $d_\ell$.
Since $f(t)\le t$, the sequence $(\gamma_i)_{i=0,\ldots,d_\ell}$ is non-increasing, and
therefore $B_a\ge\gamma_{d_\ell}\ge1-2d_\ell\sqrt\varepsilon$.
\end{proof}

\begin{proposition}
\label{prop:coercivity}
Assume that the  embedded reaction network $\cN_\Sigma(b)$ frozen at $s=b$ is complex balanced with positive equilibrium $c>0$, and assume condition ($\diamond$).  Then, there exist $\rho\in(0,\min_{j\in\Sigma}c_j)$ and $\delta>0$, such that
\begin{equation}\label{eq:5.4}
 D_\ell(z)\ge\delta A_\ell(z)
 \qquad\text{for every $\Sigma$-minimal }L_\ell,
\end{equation}
and therefore
\begin{equation}\label{eq:5.5}
 D(z)\ge\delta\max_{1\le\ell\le R} A_\ell(z),
\end{equation}
for every $z\in\R^{|\Sigma|}_{>0}$ with $\norm{z}\le\rho$.
\end{proposition}

\begin{proof}
Choose a $\Sigma$-minimal $L_\ell$.  If no estimate of the form \eqref{eq:5.4} holds in any neighbourhood of $0$, there would exist $z_n\in\R^{|\Sigma|}_{>0}$ with $z_n\to0$ for $n\to\infty$, and
\[
 \frac{D_\ell(z_n)}{A_\ell(z_n)}\to0.
\]
From Lemma~\ref{lem:uniformprop} we then have
\[
 \frac{A_a(z_n)}{A_\ell(z_n)}\to1,
 \qquad a\in L_\ell.
\]
Let
\[
 \xi(z)=\log z-\log c.
\]
Then, for every reaction $a\to a'$ in $L_\ell$,
\[
 (a'-a)\cdot\xi(z_n)
 =\log\frac{A_{a'}(z_n)}{A_a(z_n)}\to0.
\]
Hence, $u\cdot\xi(z_n)\to0$ for every $u\in S_\ell$.  By ($\diamond$), choose $0\neq v\in S_\ell\cap\R^{|\Sigma|}_{\ge0}$.  Since every coordinate of $z_n$ tends to zero while $c>0$ is fixed,
\[
 v\cdot\xi(z_n)\to-\infty,
\]
a contradiction.  Thus, for each $\Sigma$-minimal $\ell$, there are $\rho_\ell,\delta_\ell>0$ for which \eqref{eq:5.4} holds.  Finiteness of the $\Sigma$-minimal classes allows $\rho=\min_{\ell}\rho_\ell$ and $\delta=\min_{\ell}\delta_\ell$, the minima being over the $\Sigma$-minimal classes. Shrinking $\rho$ further, we may assume $\rho<\min_{j\in\Sigma}c_j$, so that $z<c$ component-wise whenever $\norm z\le\rho$.

For such $z$, put $\theta=-\xi(z)=\log(c/z)\in\R^{|\Sigma|}_{>0}$, so that $A_a(z)=e^{-a\cdot\theta}$ and
\[
 \max_{1\le\ell\le R}A_\ell(z)
 =\max_{a\in\cC^{\mathrm{act}}_\Sigma}A_a(z)=e^{-\mu},
 \qquad
 \mu=\min_{a\in\cC^{\mathrm{act}}_\Sigma}a\cdot\theta .
\]
Any complex $a^*$ attaining $\mu$ is $\Sigma$-minimal: if $a\in\cC^{\mathrm{act}}_\Sigma$ satisfies
$a\le a^*$ and $ a\neq a^*$, then $a^*- a\ge0$ is non-zero, and $\theta>0$  gives $ a\cdot\theta<\mu$, contradicting the definition of $\mu$.  Hence, $a^*$ lies in a $\Sigma$-minimal class $L_{\ell^*}$ with
$A_{\ell^*}(z)=e^{-\mu}$. Equation  \eqref{eq:5.4} applied to $L_{\ell^*}$ gives
$\delta A_{\ell^*}(z)\le D_{\ell^*}(z)\le D(z)$, which is \eqref{eq:5.5}.
\end{proof}

Condition ($\diamond$) enters the argument below only through the bounds \eqref{eq:5.4} in Proposition~\ref{prop:coercivity}.  We next give four lemmas showing that $U$ in \eqref{eq:U} is strictly increasing near the boundary, that is, $\dot U>0$.

For $s$ near $b>0$, define
\begin{equation*}
 \rho_r(s):=\frac{s^{\beta_r}}{b^{\beta_r}},
 \qquad
 \eta(s):=\max_{r\in\cR}|\rho_r(s)-1|.
\end{equation*}
Then, $\eta(s)\to0$ as $s\to b$.

\begin{lemma}
Let $U$ be as in \eqref{eq:U}. Along any positive trajectory,
\begin{equation*}
 \dot U
 =\sum_{r\in\cR}q_r\rho_r(s)A_{\alpha_r}
 \log\frac{A_{\alpha_r}}{A_{\alpha'_r}}
 =D(z)+E(z,s),
\end{equation*}
where
\begin{equation}
 E(z,s)=\sum_{r\in\cR}q_r(\rho_r(s)-1)A_{\alpha_r}
 \log\frac{A_{\alpha_r}}{A_{\alpha'_r}}.
 \label{eq:6.3}
\end{equation}
\end{lemma}

\begin{proof}
Differentiating   $U$ in \eqref{eq:U} gives
\[
 \frac{\partial U}{\partial z_j}=-\log\frac{z_j}{c_j}.
\]
Substitute the $z$-equation \eqref{eq:zeq} into $\dot U$, and use
$(\alpha'_r-\alpha_r)\cdot\log(z/c)
 =\log(A_{\alpha'_r}/A_{\alpha_r})$ together with
$k_r z^{\alpha_r}s^{\beta_r}=q_r\rho_r(s)A_{\alpha_r}$.
The claimed identity follows by separating $\rho_r(s)=1+(\rho_r(s)-1)$ and using the first expression for $D$ in \eqref{eq:4.3}. Reactions in $\Sigma$-blind linkage classes contribute zero since $\alpha_r=\alpha'_r=0$.
\end{proof}

For a $\Sigma$-active linkage class $L_\ell$, define
\begin{equation*}
 \Theta_\ell=\sum_{r\in L_\ell}q_r|A_{\alpha_r}(z)-A_{\alpha'_r}(z)|,
 \qquad
 \Psi_\ell=\sum_{r\in L_\ell}q_r(A_{\alpha_r}(z)+A_{\alpha'_r}(z)).
\end{equation*}

\begin{lemma}\label{lem:sharppert}
We have
\begin{equation}
 |E(z,s)|\le \eta(s)D(z)+\eta(s)\sum_{\ell=1}^R\Theta_\ell(z),
 \label{eq:6.5}
\end{equation}
and, for every $\Sigma$-active linkage  class,
\begin{equation}
 \Theta_\ell(z)
 \le \left(\frac98D_\ell(z)\,\Psi_\ell(z)\right)^{1/2}.
 \label{eq:6.6}
\end{equation}
\end{lemma}

\begin{proof}
For $x,y>0$, let $\Phi(x,y)=x\log(x/y)-x+y\ge0,$ and  define
\begin{equation*}
 x_r=A_{\alpha_r}(z),\qquad y_r=A_{\alpha'_r}(z)
\end{equation*}
for convenience.
Using $x\log(x/y)=\Phi(x,y)+(x-y)$ in \eqref{eq:6.3}, yields
\[
 E=\sum_{r\in\cR} q_r(\rho_r-1)\Phi(x_r,y_r)
   +\sum_{r\in\cR} q_r(\rho_r-1)(x_r-y_r).
\]
The first sum has absolute value at most
$\eta\sum_{r\in\cR} q_r\Phi(x_r,y_r)=\eta D$. The second is bounded
by $\eta\sum_{\ell=1}^R\Theta_\ell$. Blind reactions contribute
zero to both sums. This proves \eqref{eq:6.5}. For \eqref{eq:6.6}, use the elementary inequality
\begin{equation*}
 \Phi(x,y)\ge \frac89\frac{(x-y)^2}{x+y}, \qquad x,y>0,
\end{equation*}
and Cauchy--Schwarz,
\[ \Theta_\ell
 =\sum_{r\in L_\ell}
 \left(q_r\frac{(x_r-y_r)^2}{x_r+y_r}\right)^{1/2}
 \left(q_r(x_r+y_r)\right)^{1/2}
 \le \left(\frac98D_\ell\Psi_\ell\right)^{1/2}.\]
\end{proof}

The next lemma makes   use of condition ($\diamond$).

\begin{lemma}\label{lem:pertcontrol}
Under the assumptions of Proposition~\ref{prop:coercivity}, there are $\rho>0$
and $C>0$ such that, for $0<\norm z<\rho$ and $s>0$,
\begin{equation}\label{eq:pertcontrol}
 |E(z,s)|\le \eta(s)(1+C)D(z).
\end{equation}
\end{lemma}

\begin{proof}
Take $\rho,\delta$ from Proposition~\ref{prop:coercivity} and set
\[
 q^*=\max_{r\in\cR}q_r,\qquad
 r^*=\max_{1\le\ell\le R}|\{r:r\in L_\ell\}|.
\]
For $r\in L_\ell$, both $x_r$ and $y_r$ (as in Lemma \ref{lem:sharppert}) are at most $A_\ell$.
Thus, by \eqref{eq:5.5},
\begin{equation*}
 \Psi_\ell(z)\le 2q^*r^*A_\ell(z)
 \le 2q^*r^*\max_{1\le\ell'\le R}A_{\ell'}(z)
 \le \frac{2q^*r^*}{\delta}D(z),
\end{equation*}
for every $\ell$, minimal or not.
Lemma~\ref{lem:sharppert} therefore yields
\[
 \Theta_\ell(z)\le
 \left(\frac{9q^*r^*}{4\delta}\right)^{1/2}
 \bigl(D_\ell(z)D(z)\bigr)^{1/2}.
\]
Summing over $\ell$, using Cauchy--Schwarz in the form
$\sum_{\ell=1}^R D_\ell^{1/2}\le\bigl(R\sum_{\ell=1}^R D_\ell\bigr)^{1/2}$ and
$D=\sum_{\ell=1}^R D_\ell$, gives
\begin{equation*}
 \sum_{\ell=1}^R\Theta_\ell(z)\le CD(z),\qquad
 C=\left(\frac{9q^*r^*R}{4\delta}\right)^{1/2}.
\end{equation*}
The bound \eqref{eq:pertcontrol} now follows from \eqref{eq:6.5}.
\end{proof}

\begin{lemma}\label{lem:strictincrease}
Under the assumptions of Proposition~\ref{prop:coercivity}, there are
$\rho>0$ and $\epsilon>0$ such that
\begin{equation*}
 \dot U(z,s)\ge\frac12D(z)
 \ge\frac\delta2\max_{1\le\ell\le R} A_\ell(z)>0
\end{equation*}
whenever $z>0$, $0<\norm z<\rho$ and $\norm{s-b}<\epsilon$. By setting
$$\mathcal U=\{(z,s): z>0,\ 0<\norm z<\rho,\ \norm{s-b}<\epsilon\},$$
 then $\dot U>0$ throughout $\mathcal U$.
\end{lemma}

\begin{proof}
Choose $\rho,\delta,C$ as in Proposition~\ref{prop:coercivity}
and Lemma~\ref{lem:pertcontrol}. Since $\eta(s)\to0$ as $s\to b$,
choose $\epsilon>0$ such that $\eta(s)(1+C)\le1/2$
when $\norm{s-b}<\epsilon$. By the identity for $\dot U$ and
\eqref{eq:pertcontrol},
\[
 \dot U=D+E\ge (1-\eta(s)(1+C))D\ge\frac12D.
\]
The remaining inequalities follow from Proposition~\ref{prop:coercivity}.
\end{proof}

\section{Proof of Theorem~\ref{thm:boundary}}\label{sec:th1.1}

Extend $U$ in \eqref{eq:U} continuously to $z\ge0$ by $0\log0=0$.  By
Proposition~\ref{prop:coercivity}, $\rho<\min_{j\in\Sigma}c_j$, so $\log(z_j/c_j)<0$ for
$0<z_j\le\rho$.  Every term of $U$ is therefore non-negative on $\{z\ge0:\norm z\le\rho\}$, and
vanishes only if the corresponding coordinate does.  Hence $U(z)>0$ for $z\ge0$, $z\neq0$,
$\norm z\le\rho$, and by compactness of $\{z\ge0:\norm z=\rho\}$,
\begin{equation}\label{eq:nu}
 \nu:=\min\{U(z): z\ge0,\ \norm z=\rho\}>0 .
\end{equation}

\begin{proof}[Proof of Theorem~\ref{thm:boundary}]
Choose $\rho$ and $\epsilon$ as in Lemma~\ref{lem:strictincrease}, and let $T_0$ be such that
$\norm{s(t)-b}<\epsilon$ for $t\ge T_0$.  Since $z(t)>0$ at all finite times, the trajectory lies
in $\mathcal U$ at every time $t\ge T_0$ with $\norm{z(t)}<\rho$, and $U(z(\cdot))$ is
non-decreasing on every interval of such times, by Lemma~\ref{lem:strictincrease}.

Let $t\ge T_0$ with $\norm{z(t)}<\rho$.  If $\norm{z(\tau)}=\rho$ for some $\tau\in[T_0,t]$, let
$\sigma$ be the largest such $\tau$.  Then $\norm{z(\tau)}<\rho$ for $\tau\in(\sigma,t]$, so the
trajectory lies in $\mathcal U$ on $(\sigma,t]$ and, by continuity,
\[
 U(z(t))\ \ge\ U(z(\sigma))\ \ge\ \nu
\]
by \eqref{eq:nu}.  If there is no such $\tau$, then $\norm{z(\tau)}<\rho$ throughout $[T_0,t]$,
the trajectory lies in $\mathcal U$ on that whole interval, and $U(z(t))\ge U(z(T_0))>0$.  In
either case, let
\begin{equation*}
\nu_0=\left\{ \begin{array}{cl} \min\{\nu,\,U(z(T_0))\}, & \text{for}\quad\norm{z(T_0)}<\rho,\\
\nu,&  \text{for}\quad\norm{z(T_0)}\ge\rho,\end{array}\right.
\end{equation*}
then
\[
 U(z(t))\ \ge\ \nu_0>0. 
\]

As $U$ is continuous on $\{z\ge0\}$ with $U(0)=0$, there is $\lambda>0$ such that $U(z)<\nu_0$
whenever $z\ge0$ and $\norm z<\lambda$.  Consequently $\norm{z(t)}\ge\lambda$ for every $t\ge T_0$
with $\norm{z(t)}<\rho$, while $\norm{z(t)}\ge\rho$ for the remaining $t\ge T_0$.  Hence
\[
 \liminf_{t\to\infty}\norm{z(t)}\ \ge\ \min\{\lambda,\rho\}>0 ,
\]
so no sequence $t_n\to\infty$ satisfies $z(t_n)\to0$.  Thus $(0,b)\notin\omega(x_0)$, and in
particular \eqref{eq:zs} cannot hold.
\end{proof}

\section{Boundary equilibria and frozen complex balance}\label{sec:frozen}

Throughout this section and the next, the reaction network $\cN$ is weakly reversible and complex
balanced at $x^*=(z^*,s^*)$, and $x(\cdot)$ is the positive trajectory started at $x_0>0$.  We write
\[
 \widetilde S=\operatorname{span}\{y'_r-y_r:r\in\cR\}\subseteq\R^m
\]
for the stoichiometric space of $\cN$, retaining $S\subseteq\R^{|\Sigma|}$ for the
embedded one.  If $w\in\omega(x_0)$ lies on the boundary, its zero set is a
siphon~\cite{AngeliDeLeenheerSontag2007}. All the objects below are attached to such a
siphon $\Sigma$, and we suppress $\Sigma$ from the notation when it is fixed.

The $\Sigma$-blind reactions involve only the species in $\Sigma^c$, by \eqref{eq:zeq}
and $\alpha_r=\alpha'_r=0$.  Viewed as a mass-action reaction network on those species,
with complexes $\beta_r$ and rate constants $k_r$, they form the \emph{$\Sigma$-blind
subsystem}. Its stoichiometric space is
\[
 S_{\mathrm{bl}}(\Sigma)=\operatorname{span}\{\beta'_r-\beta_r:
 r\ \text{$\Sigma$-blind}\}.
\]
These are the only reactions with non-zero rate on the face $z=0$, and, since
$\Sigma$ is a siphon, $\dot z=0$ there. Hence, the face is invariant and the flow on it is
the flow of the $\Sigma$-blind subsystem.  Finally, let
\[
 S_0(\Sigma)=\{v\in\R^{|\Sigma^c|}:(0,v)\in\widetilde S\},
\]
the space of stoichiometric vectors internal to the face.  A $\Sigma$-blind reaction has
reaction vector $(0,\beta'_r-\beta_r)\in\widetilde S$, so
\begin{equation}\label{eq:blincl}
 S_{\mathrm{bl}}(\Sigma)\subseteq S_0(\Sigma)
\end{equation}
always, and the condition imposed below asks only for the reverse inclusion.

\begin{lemma}\label{lem:omega}
Let $w=(0,b)\in\omega(x_0)$ with $b>0$ and zero set $\Sigma$.  Then, $w$ is an
equilibrium, the $\Sigma$-blind subsystem is complex balanced at $b$, and
\begin{equation*}
 \log(b/s^*)\perp S_{\mathrm{bl}}(\Sigma).
\end{equation*}
\end{lemma}

\begin{proof}
A trajectory of a complex balanced system converges to the set of equilibria (Siegel and
MacLean~\cite{SiegelMacLean2000} and Sontag~\cite{Sontag2001}), so $w$ is an equilibrium.
The only reactions with non-zero rate at $w$ are the $\Sigma$-blind ones, so $b$ is a positive
equilibrium of the $\Sigma$-blind subsystem.  That subsystem is complex balanced at $s^*$, because
a $\Sigma$-blind complex interacts only in $\Sigma$-blind reactions, so the complex balance
relations of $\cN$ at those complexes involve $\Sigma$-blind reactions only.  By Horn and
Jackson~\cite{HornJackson1972}, every positive equilibrium of a complex balanced system is complex
balanced, the set of them being $\{s>0:\log(s/s^*)\perp S_{\mathrm{bl}}(\Sigma)\}$.
\end{proof}

Lemma~\ref{lem:omega} is in substance known. A closely related statement for boundary
equilibria, assuming deficiency zero rather than complex balance, is Theorem~3.4 of Cappelletti and
Wiuf~\cite{CappellettiWiuf2016}.

Theorem~\ref{thm:boundary} requires the embedded reaction network frozen at $b$ to be complex
balanced.  We give three sufficient conditions.

\begin{proposition}
\label{prop:extensionCB}
Let $h=\log(b/s^*)$.  If there exists $p\in\R^{|\Sigma|}$ with
\[
 (p,h)\in\widetilde S^\perp,
\]
then $\cN_\Sigma(b)$ frozen at $b$ is complex balanced at $c=z^*\circ e^p$.
\end{proposition}

\begin{proof}
For a complex $y=(\alpha,\beta)$,
\[
 c^\alpha b^\beta=(z^*)^\alpha(s^*)^\beta\,
 e^{\alpha\cdot p+\beta\cdot h}.
\]
Since $(p,h)\perp\widetilde S$, the exponential factor is constant on each
linkage class of $\cN$.  Thus, the original complex balanced circulation $J^*$ is rescaled by a
positive constant per linkage class, and remains a circulation.  Summing the node balances
over original complexes with the same $\Sigma$-projection gives complex balance of the
embedded system frozen at $b$, at the point $c$.
\end{proof}

\begin{corollary}
\label{cor:structuralCB}
If
\begin{equation}\label{eq:structuralCB}
 S_{\mathrm{bl}}(\Sigma)=S_0(\Sigma),
\end{equation}
then $\cN_\Sigma(b)$ frozen at every boundary $\omega$-limit point $(0,b)$ with zero
set $\Sigma$ is complex balanced.  Condition \eqref{eq:structuralCB} holds in particular
when $S_0(\Sigma)=\{0\}$, that is, when a stoichiometric compatibility class intersects the
face $z=0$ 
 in at most one point.
\end{corollary}

\begin{proof}
Lemma~\ref{lem:omega} gives $h=\log(b/s^*)\perp S_{\mathrm{bl}}(\Sigma)$, hence
$h\perp S_0(\Sigma)$ by \eqref{eq:structuralCB}.  Now
\begin{equation*}
 \pi_{\Sigma^c}\bigl(\widetilde S^\perp\bigr)=S_0(\Sigma)^\perp
\end{equation*}
(where $\pi_{\Sigma^c}$ denotes the projection), because for $v\in\R^{|\Sigma^c|}$,
\[
 v\perp\pi_{\Sigma^c}(\widetilde S^\perp)
 \iff (0,v)\perp\widetilde S^\perp
 \iff (0,v)\in\widetilde S
 \iff v\in S_0(\Sigma),
\]
and one takes orthogonal complements.  Thus there is $p$ with $(p,h)\in\widetilde
S^\perp$, and Proposition~\ref{prop:extensionCB} applies.  Finally, if $S_0(\Sigma)=\{0\}$
then two points of one compatibility class lying on the face differ by some
$(0,v)\in\widetilde S$ and therefore coincide; and $S_{\mathrm{bl}}(\Sigma)=\{0\}$ by
\eqref{eq:blincl}, so \eqref{eq:structuralCB} holds.
\end{proof}

The third criterion is point-wise in the rate vector of the embedded reaction network $\cN_\Sigma(b)$, and allows the rate constants of the original network $\cN$ to change.

\begin{proposition}
\label{prop:ratematch}
Fix the  rate vector $k=(k_r)_{r\in\cR}$ of $\cN$ and let $b>0$.  Suppose there are a positive rate vector
$\kappa=(\kappa_r)_{r\in\cR}$ and a positive complex balanced equilibrium $(c,s)$ of $\cN$ with rate
vector $\kappa$, such that
\begin{equation}\label{eq:ratematch}
 \kappa_r s^{\beta_r}=k_r b^{\beta_r},\qquad r\in\cR .
\end{equation}
Then, the embedded network $\cN_\Sigma(b)$ with rate vector $k$ is complex balanced at $c$.
\end{proposition}

\begin{proof}
The equilibrium fluxes of $\cN$ with rates $\kappa$ are
$\kappa_rc^{\alpha_r}s^{\beta_r}=k_rb^{\beta_r}c^{\alpha_r}$, and they satisfy node
balance on the original reaction graph.  Summing those balances over the complexes of $\cN$
with the same $\Sigma$-projection gives node balance for the embedded reaction graph with the
reaction rates $k_rb^{\beta_r}$, at the point $c$.
\end{proof}

\begin{example}
Consider
\[
 Z+S_1\longrightarrow 2Z+S_2\longrightarrow 3Z+S_3
 \longrightarrow Z+S_1 .
\]
The reaction network is weakly reversible, with three complexes, one linkage class and
stoichiometric rank two, hence deficiency zero.  With $\kappa_1=\kappa_2=\kappa_3=1$ the
point $(c,s_1,s_2,s_3)=(1,1,1,1)$ is complex balanced.  Hence, by
Proposition~\ref{prop:ratematch}, the frozen embedded reaction network  is complex balanced at $c=1$,
whenever
\[
 k_1b_1=k_2b_2=k_3b_3=1 .
\]
Here, $\Sigma=\{Z\}$ is a siphon and the embedded reaction network is
\[
 Z\longrightarrow2Z\longrightarrow3Z\longrightarrow Z,
\]
with stoichiometric rank one and deficiency $3-1-1=1$.  The  criterion therefore
applies in cases where the embedded reaction network is not deficiency zero.  Condition
($\diamond$) also holds here, since there is a single $\Sigma$-active embedded linkage
class.
\end{example}

\begin{example}
For the same embedded cycle with frozen rates $\lambda_1,\lambda_2,\lambda_3$, complex
balance at $c>0$ requires $\lambda_1c=\lambda_2c^2=\lambda_3c^3$, hence
$\lambda_2^2=\lambda_1\lambda_3$.  For $(\lambda_1,\lambda_2,\lambda_3)=(1,1,2)$ the frozen
system is not complex balanced, so no representation \eqref{eq:ratematch} exists for that
frozen rate vector.
\end{example}

\section{The global attractor conclusion}\label{sec:gac}

Condition \eqref{eq:structuralCB} has a second consequence, independent of the embedded
reaction network: it makes the boundary equilibria isolated inside a compatibility class.

Only those siphons that can occur as zero sets matter.  Call a subset
$\Sigma\subseteq\{1,\ldots,m\}$ \emph{realisable} if some point of the closure of the
stoichiometric compatibility class of $x_0$ has zero set exactly $\Sigma$, and let $\mathcal
S(x_0)$ be the set of non-empty realisable siphons. Since $\omega(x_0)$ lies in the closure of the stoichiometric compatibility class of $x_0$, then zero set of any boundary point of $\omega(x_0)$ belongs to
$\mathcal S(x_0)$.  If $\cN$ has a conservation law with positive coefficients supported on a set
$I$, then no realisable $\Sigma$ contains $I$. This is what keeps $\mathcal S(x_0)$ small in the
examples of Section~\ref{sec:examples}.

\begin{corollary}
\label{cor:singleton}
Assume \eqref{eq:structuralCB} for every $\Sigma\in\mathcal S(x_0)$.  Then,
either $x(t)\to x^*$, or $x(t)$ converges to a single boundary equilibrium $(0,b)$ with
$b>0$ whose zero set is a siphon.
\end{corollary}

\begin{proof}
By Proposition~\ref{prop:dichotomy}, we may assume $\omega(x_0)\subseteq\partial\R^m_{\ge0}$
and that every point of $\omega(x_0)$ is an equilibrium; each such point has a zero set
in $\mathcal S(x_0)$, and $\mathcal S(x_0)$ is finite.  Fix $\Sigma\in\mathcal S(x_0)$, and let
$(0,b_1),(0,b_2)\in\omega(x_0)$ both have zero set $\Sigma$.  By Lemma~\ref{lem:omega},
$b_1$ and $b_2$ are positive equilibria of the $\Sigma$-blind subsystem, which is complex
balanced at $s^*$. By Horn and Jackson~\cite{HornJackson1972}, the positive equilibria of a
complex balanced system intersect each coset of its stoichiometric space in exactly one point,
here in exactly one point of each coset of $S_{\mathrm{bl}}(\Sigma)$.  On the other hand
$(0,b_1)-(0,b_2)\in\widetilde S$, so $b_1-b_2\in S_0(\Sigma)=S_{\mathrm{bl}}(\Sigma)$ by
\eqref{eq:structuralCB}: the two points lie in the same such coset, whence $b_1=b_2$.
Therefore, $\omega(x_0)$ is finite, and being connected it is a single point.
\end{proof}

\begin{theorem}
\label{thm:GAC}
Let $\cN$ be a weakly reversible and complex balanced reaction network, and let $x_0>0$.  Assume that for
every $\Sigma\in\mathcal S(x_0)$,
\begin{enumerate}
\item[(a)] $S_0(\Sigma)=S_{\mathrm{bl}}(\Sigma)$, and
\item[(b)] condition ($\diamond$) holds for every $\Sigma$-minimal linkage class of the embedded reaction network.
\end{enumerate}
Then, $x(t)\to x^*$.
\end{theorem}

\begin{proof}
By (a) and Corollary~\ref{cor:singleton}, either $x(t)\to x^*$, and we are done, or
$x(t)$ converges to a single boundary equilibrium $(0,b)$ with $b>0$ and zero set a siphon
$\Sigma\in\mathcal S(x_0)$, that is \eqref{eq:zs} holds.  In the latter case (a) and
Corollary~\ref{cor:structuralCB} make the embedded reaction network frozen at $b$ complex balanced,
and (b) supplies ($\diamond$). Theorem~\ref{thm:boundary} then gives
$\liminf_{t\to\infty}\norm{z(t)}>0$, contradicting \eqref{eq:zs}.
\end{proof}

\begin{remark}
Under   hypothesis (a), suppose that $z(t_n)\to0$ for some sequence $t_n\to\infty$, and that $\liminf_{t\to\infty}s_j(t)>0$ for every surviving species $j$. By boundedness, after passing to a subsequence, $s(t_n)\to b>0$. Corollary 7.1 implies that the boundary $\omega$-limit set is a singleton. Consequently, $s(t)\to b$, and Theorem 1.2 yields a contradiction.
\end{remark}

\section{Examples}\label{sec:examples}

The first example shows that ($\diamond$) is strictly weaker than requiring
$S_\ell\cap\R^{|\Sigma|}_{\ge0}\neq\{0\}$ for every $\Sigma$-active embedded linkage class.

\begin{example}\label{ex:minimalmatters}
Consider the four-species reaction network
\[
 A+B\rightleftharpoons 2A,
 \qquad 3A+C\rightleftharpoons 3A+D,
 \qquad C\rightleftharpoons D ,
\]
with all six rate constants equal to one.  It is weakly reversible with six complexes, three
linkage classes and stoichiometric space
$\widetilde S=\operatorname{span}\{(1,-1,0,0),(0,0,-1,1)\}$ of dimension two. Hence the deficiency is
$6-3-2=1$, and it is complex balanced at $x^*=(1,1,1,1)$.  The conserved quantities are $T_1=x_A+x_B$
and $T_2=x_C+x_D$.

A realisable set therefore contains at most one of $A,B$ and at most one of $C,D$, leaving the
eight sets $\{A\},\{B\},\{C\},\{D\}$ and $\{A,C\},\{A,D\},\{B,C\},\{B,D\}$.  Each of the last
seven contains $B$, $C$ or $D$ while being disjoint from the source of a reaction producing that
species, namely $2A\to A+B$, $D\to C$ and $C\to D$ respectively, and so is not a siphon.  The
reaction network does have other siphons, namely $\{A,B\}$, $\{C,D\}$, $\{A,C,D\}$ and
$\{A,B,C,D\}$, but each of them contains $\{A,B\}$ or $\{C,D\}$ and is therefore not realisable.
Hence, $\mathcal S(x_0)=\{\{A\}\}$.

For $\Sigma=\{A\}$, the linkage class $C\rightleftharpoons D$ is $\Sigma$-blind, so
$S_{\mathrm{bl}}(\Sigma)=\operatorname{span}\{(0,-1,1)\}$, while an element of $\widetilde S$
with vanishing $A$-coordinate is a multiple of $(0,0,-1,1)$. Hence, $S_0(\Sigma)$ is the same space and (a) holds.

The embedded reaction network has the two $\Sigma$-active linkage classes
\[L_1\colon \ A\rightleftharpoons 2A, \qquad L_2\colon\ 3A\rightleftharpoons 3A,\]
the second a single complex carrying two parallel loops, and one $\Sigma$-blind class,
$0\rightleftharpoons 0$.  Here $S_2=\{0\}$, so
 $S_2\cap\R_{\ge0}=\{0\}$, and the requirement $S_\ell\cap\R^{|\Sigma|}_{\ge0}\neq\{0\}$ fails for $L_2$. So does  the bound \eqref{eq:5.4} for $L_2$, as   $D_2\equiv0$, while $A_2(z)=(z/c)^3>0$.

In contrast, $\cC^{\mathrm{act}}_\Sigma=\{A,2A,3A\}$ has only one $\Sigma$-minimal element, namely $A$, hence $L_2$ is not $\Sigma$-minimal. Moreover,  $S_1=\operatorname{span}\{1\}$ gives ($\diamond$) for $L_1$.
Theorem~\ref{thm:GAC} applies, but would not if ($\diamond$) were imposed on all $\Sigma$-active classes.
\end{example}

 The next example is not covered by~\cite{GopalkrishnanMillerShiu2014}.

\begin{example}\label{ex:notse}
Consider the four-species network
\[
 A+C\rightleftharpoons B+C,\qquad 2A+C\rightleftharpoons 2A+D .
\]
It is weakly reversible with four complexes, two linkage classes and stoichiometric space
$\widetilde S=\operatorname{span}\{(-1,1,0,0),(0,0,-1,1)\}$ of dimension two, hence of deficiency
$4-2-2=0$. Therefore, it is  complex balanced for every choice of rate constants.  Neither half is
autonomous: the equations for $x_A,x_B$ involve $x_C$, and those for $x_C,x_D$ involve $x_A$.

The conserved quantities are $T_1=x_A+x_B$ and $T_2=x_C+x_D$, so, as in
Example~\ref{ex:minimalmatters}, a realisable set contains at most one of $A,B$ and at most one of
$C,D$; of those eight sets only $\{A,C\}$ is a siphon, the set $\{A\}$,  for instance, fails
because $B+C\to A+C$ produces $A$ from a complex not containing it.  Every other siphon of the
reaction network contains $\{A,B\}$ or $\{C,D\}$ and is therefore not realisable.  Hence
$\mathcal S(x_0)=\{\{A,C\}\}$.

For $\Sigma=\{A,C\}$, no complex has zero $\Sigma$-projection, so there are no $\Sigma$-blind linkage classes and $S_{\mathrm{bl}}(\Sigma)=\{0\}$. Furthermore,  an element of $\widetilde S$ with
vanishing $A$- and $C$-coordinates is zero, so $S_0(\Sigma)=\{0\}$, and (a) holds.

The embedded reaction network is
\[
 A+C\rightleftharpoons C,
 \qquad
 2A+C\rightleftharpoons 2A,
\]
with two $\Sigma$-active linkage classes.  Its $\Sigma$-minimal complexes are $C$ and $2A$, so
both classes are $\Sigma$-minimal. Their stoichiometric spaces, $S_1=\operatorname{span}\{(1,0)\}$ and $S_2=\operatorname{span}\{(0,1)\}$, intersect $\R^2_{\ge0}$ non-trivially, so ($\diamond$) holds.  Theorem~\ref{thm:GAC} applies.

However, the reaction network is not strongly endotactic.  Take $w=(-1,-1,0,-1)$, which is not orthogonal to
$\widetilde S$, because $w\cdot(0,0,-1,1)=-1$.  The values of $w$ on the complexes are
\[
 w\cdot(A{+}C)=w\cdot(B{+}C)=-1,\qquad w\cdot(2A{+}C)=-2,\qquad w\cdot(2A{+}D)=-3 ,
\]
so the complexes maximising $w\cdot y$ are $A+C$ and $B+C$, and the two reactions between them
satisfy $w\cdot(y'-y)=0$.  Hence, there is no reaction $y\to y'$ with $y$ maximising $w\cdot y$
over all complexes and $w\cdot y'<w\cdot y$, which is what the definition of a strongly
endotactic reaction network requires for every $w$ not orthogonal to $\widetilde S$.  (The
reaction $2A+C\to2A+D$ does satisfy $w\cdot y'<w\cdot y$, but its source is not $w$-maximal.)
\end{example}

The hypotheses of Theorem~\ref{thm:GAC} and strong endotacticity are incomparable; the opposite
direction is illustrated next.

\begin{example}\label{ex:incomparable}
Consider
\[
 A+B\longrightarrow A+C\longrightarrow 2A\longrightarrow A+B,
\]
which is weakly reversible with three complexes, one linkage class and stoichiometric rank two,
hence of deficiency zero and complex balanced for all rate constants. Being weakly reversible with
a single linkage class, it is strongly endotactic~\cite{GopalkrishnanMillerShiu2014}.  Every complex contains $A$,
so $\Sigma=\{A\}$ is a siphon; and it is realisable, since the compatibility class
$\{x_A+x_B+x_C=T\}$ contains points with $x_A=0$ and $x_B,x_C>0$.
 There are  no $\Sigma$-blind linkage classes, so $S_{\mathrm{bl}}(\Sigma)=\{0\}$,
whereas $S_0(\Sigma)=\operatorname{span}\{(-1,1)\}$.  Hypothesis (a) of Theorem~\ref{thm:GAC}
therefore fails.
\end{example}


\end{document}